\documentclass[preprint,12pt]{elsarticle}

\usepackage{amsmath,amsthm,amssymb}
\usepackage{enumitem} 

\usepackage[colorlinks]{hyperref}

\usepackage{mathtools}

\theoremstyle{plain}
\newtheorem{theorem}{Theorem}
\newtheorem{corollary}[theorem]{Corollary}
\newtheorem{proposition}[theorem]{Proposition}

\theoremstyle{definition}
\newtheorem{definition}[theorem]{Definition}

\newtheorem{example}[theorem]{Example}

\theoremstyle{remark}
\newtheorem{remark}[theorem]{Remark}

\newcommand{\ellZ}{\ell_{\text{\upshape Z}}}

\newcommand{\norm}[1]{\lVert #1 \rVert}
\newcommand{\Norm}[1]{\left\lVert #1 \right\rVert}
\newcommand{\abs}[1]{\lvert #1 \rvert}

\DeclareMathOperator{\diag}{diag}
\DeclareMathOperator{\real}{Re}
\DeclareMathOperator{\imag}{Im}

\journal{Systems \& Control Letters}

\begin{document}

\begin{frontmatter}



\title{Efficient method for computing lower bounds on the $p$-radius of switched linear systems}


\author[upenn]{Masaki Ogura\corref{cor}}
\ead{ogura@seas.upenn.edu}

\author[upenn]{Victor M. Preciado}
\ead{preciado@seas.upenn.edu}

\author[ucl,fnrs]{Rapha\"el M.~Jungers}
\ead{raphael.jungers@uclouvain.be}

\cortext[cor]{Corresponding author}

\address[upenn]{Department of Electrical and Systems Engineering,
University of Pennsylvania, 
Philadelphia, PA 19104, USA}

\address[ucl]{ICTEAM Institute, Universit\'e catholique de Louvain, 4 avenue Georges Lemaitre, B-1348 Louvain-la-Neuve, Belgium}

\address[fnrs]{F.R.S.-FNRS Research Associate.}

\begin{abstract}
This paper proposes lower bounds on a quantity called $L^p$-norm joint spectral
radius, or in short, $p$-radius, of a finite set of matrices. Despite its wide
range of applications to, for example, stability analysis of switched linear
systems and the equilibrium analysis of switched linear economical models,
algorithms for computing the $p$-radius are only available in a very limited
number of particular cases. The proposed lower bounds are given as the spectral
radius of an average of the given matrices weighted via Kronecker products and
do not place any requirements on the set of matrices. We show that the proposed
lower bounds theoretically extend and also can practically improve the existing
lower bounds. A Markovian extension of the proposed lower bounds is also
presented.
\end{abstract}


\end{frontmatter}




\section{Introduction}

The \emph{$L^p$-norm joint spectral radius}~\cite{Lau1995,Protasov1997} (often
called {\it $p$-radius}) of {an indexed family} of $n\times n$ real matrices
$\mathcal M = \{A_1, \dotsc, A_N\}$ is defined by, for $p\geq 1$,
\begin{equation*}
\rho_p(\mathcal M) =
\lim_{k\to\infty} \left( \frac{\sum_{i_1,\dotsc,i_k =1}^N
\norm{A_{i_k}\dotsm A_{i_1}}^p}{N^k} \right)^{\!1/(pk)}, 
\end{equation*}
where $\norm{\cdot}$ denotes the maximum singular value of a matrix. Introduced
by Jia~\cite{Jia1995} and Wang~\cite{Wang1996} independently for $p=1$, the
$p$-radius now plays an important role in various fields of applied mathematics.
A classical application of the $p$-radius is in the characterization of the
regularity of wavelet functions in $L^p$ spaces~\cite{Wang1996,Lau1995}. The
mean stability of a class of switched linear systems is determined by the value
of the $p$-radius~\cite{Ogura2012b,Jungers2010}. The so-called indeterminacy of
a switching linear economic model can be checked through the
$p$-radius~\cite{Barthelemy2013a}.

However, being defined as a limit of a sequence on matrix products, the
$p$-radius is known to be difficult to compute. There is no formula available
for its computation except under the conditions that $p$ is an even integer or
that the given set of matrices leaves a common proper cone
invariant~\cite{Protasov1997,Ogura2012b}. Although for the latter case there
exists a converging approximation method~\cite{Jungers2011b} that does not
require $p$ to be an integer, in a general case, even approximating the
$p$-radius is an NP-hard problem~\cite{Jungers2011b}. As for bounds on the
$p$-radius, the sequence defining the $p$-radius is decreasing and therefore
gives upper bounds, though their computation requires exponentially growing
costs. Finally, the lower bounds on the $p$-radius in the literature~\cite{Zhou1998a,Barthelemy2013a} are often not very accurate.

In this paper we propose novel lower bounds on the $p$-radius for integer values
of $p$ with no assumptions on the given set of matrices on the contrary
to~\cite{Protasov1997,Jungers2011b}. The lower bounds are given as the spectral
radius of a weighted average of the given matrices and the weights are realized
by Kronecker product of matrices. We will show that, with appropriately chosen
weighting matrices, the proposed bounds extend and also can improve the lower
bounds in the literature~\cite{Zhou1998a,Barthelemy2013a}. The obtained results
are furthermore generalized to the Markovian setting. This in particular enables
us to use the $p$-radius to study the stability of so-called Markov jump linear
systems~\cite{Costa2005}, which are switched linear systems whose parameter
changes by following a time-homogeneous Markov chain. The generalization is
based on a stochastic counterpart of the so-called $\Omega$-lift of
matrices~\cite{Kozyakin2014}.

The organization of this paper is as follows. In Section~\ref{sec:NovelLBB}, we
derive a novel lower {bound} for the $p$-radius. In Section~\ref{sec:markov}, we
provide a Markovian extension. The notations used in this paper are standard.
{The identity matrix is denoted by $I$.} The spectral radius of a square matrix
is denoted by~$\rho(\cdot)$. The Kronecker product (see, e.g.,
\cite{Brewer1978}) of matrices $A$ and $B$ is denoted by~$A\otimes B$.
{$\mathcal M$ denotes an indexed family $\{ A_1, \dotsc, A_N\}$ of $n\times n$
real matrices.} Finally, for ease of reference, we list some important
properties of the $p$-radius in the following proposition:

\begin{proposition}\label{prop:ease}
Let $p\geq 1$ be a positive integer. 
\begin{enumerate}
\item The sequence $\{h_k (\mathcal M)\}_{k=1}^\infty$ defined by
\begin{equation*}
h_k(\mathcal M)= \left( \frac{\sum_{i_1,\dotsc,i_k =1}^N
\norm{A_{i_k}\dotsm A_{i_1}}^p}{N^k} \right)^{\!1/(pk)}
\end{equation*}
is decreasing \cite{Jungers2011b}. 

\item  If either $p$ is even or $\mathcal M$ leaves a proper cone invariant,
then $\rho_p(\mathcal M) = \rho\big( N^{-1}\sum_{i=1}^N A_i^{\otimes
p}\big)^{1/p}$, where $A^{\otimes p}$ denotes the Kronecker product of $p$
copies of $A$ (see \cite{Protasov1997}).

\item It holds that $\rho_p(A_1, \dotsc, A_N) = \bigl( \rho_{1}(A_1^{\otimes p},
\dotsc, A_N^{\otimes p}) \bigr)^{1/p}$ (see \cite{Jungers2011b}).
\end{enumerate}
\end{proposition}

\begin{remark}
Throughout the paper, we often omit the dependence of quantities on the
underlying {family} $\mathcal M$ of matrices when it is clear from the context.
Precisely speaking, for a function $f$ defined on $(\mathbb{R}^{n\times n})^N$,
we may simply write $f(\mathcal M)$ as $f$. Also, abusing notation, we sometimes
write $f(\mathcal M)$ as $f(A_1, \dotsc, A_N)$ when $\mathcal M = \{A_1, \dotsc,
A_N\}$.
\end{remark}

\section{Novel lower bounds}\label{sec:NovelLBB}

In this section, we present novel lower bounds on the $p$-radius. We also prove
that these new bounds outperform existing lower bounds. Notice that, according
to the third claim of Proposition~\ref{prop:ease}, any result on the $1$-radius
($p=1$) is directly applicable to the general $p$-radius; thus, we shall focus
on the particular case $p=1$ for the rest of the paper. In order to state our
main results, we need to recall the definition of the joint spectral
radius~\cite{Jungers2009}:
\begin{equation*}
\rho_\infty(\mathcal M)
=
\adjustlimits \limsup_{k\to\infty} \max_{i_1, \dotsc, i_k \in \{1, \dotsc, N\}}
\norm{A_{i_k} \dotsm A_{i_1}}^{1/k}. 
\end{equation*}
Our first result is stated in the next theorem:

\begin{theorem}\label{thm:main}
For $\mathcal W = \{W_1, \dotsc, W_N\} \subset \mathbb{R}^{m\times m}$, let
\begin{equation*}
\lambda_{ \mathcal W}(\mathcal M) 
= 
\rho\left(
\frac{1}{N}\sum_{i=1}^N W_i \otimes A_i
\right). 
\end{equation*}
If $\rho_\infty(\mathcal W) = 1$, then $\lambda_{\mathcal W}(\mathcal M)\leq
\rho_1(\mathcal M)$.
\end{theorem}

\begin{proof}
Using Gelfand's formula, we obtain
\begin{equation}\label{eq:Gelfand}
\lambda_{\mathcal W}(\mathcal M)
=
\lim_{k\to\infty} \Norm{
\left(\frac{1}{N} \sum_{i=1}^N W_i \otimes A_i \right)^k
}^{1/k}. 
\end{equation}
Using general identities \cite{Brewer1978} about Kronecker products: 
\begin{align}
(A\otimes C) (B\otimes D) &= (AB)\otimes (CD), 
\label{eq:ABCD}
\\
\norm{A\otimes B} &= \norm{A}\,\norm{B}, 
\label{eq:|AoxB|}
\end{align}
we can evaluate the norm in \eqref{eq:Gelfand} as
\begin{equation*}
\begin{aligned}
\Norm{
\left(\frac{1}{N} \sum_{i=1}^N W_i \otimes A_i \right)^k
}
&=
\Norm{
\sum_{i_1, \dotsc, i_k = 1}^N\frac{{(W_{i_k} \otimes A_{i_k}) \dotsm ( W_{i_1} \otimes A_{i_1})}}{N^k} 
}
\\
&\leq
\sum_{i_1, \dotsc, i_k=1}^N \frac{\norm{(W_{i_k} \otimes A_{i_k}) \dotsm ( W_{i_1} \otimes A_{i_1})}}{N^k}
\\
&=
\sum_{i_1, \dotsc, i_k = 1}^N \frac{ \norm{W_{i_k} \dotsm W_{i_1}} \,\norm{A_{i_k} \dotsm A_{i_1}}}{N^k}
\\
&\leq
\left(\max_{i_1, \dotsc, i_k}\norm{W_{i_k} \dotsm W_{i_1}} \right)
\frac{\sum_{i_1, \dotsc, i_k} \norm{A_{i_k} \dotsm A_{i_1}}}{N^k}. 
\end{aligned}
\end{equation*}
Taking the power of $1/k$ of the last expression, and substituting in
\eqref{eq:Gelfand}, we obtain
\begin{equation}\label{eq:fromtheproof}
\lambda_{\mathcal W}(\mathcal M) \leq
\rho_\infty(\mathcal W) \rho_1(\mathcal M) = \rho_1(\mathcal M), 
\end{equation}
as desired.
\end{proof}

\begin{remark}\label{rem:rho_infty<=1}
From \eqref{eq:fromtheproof}, we can see that the equality constraint
$\rho_\infty(\mathcal W) = 1$ in Theorem~\ref{thm:main} can be relaxed to the
inequality constraint $\rho_\infty(\mathcal W) \leq 1$. Although the largest
lower bound is clearly attained when the joint spectral radius of $\mathcal W$
is equal to one, it is convenient in practice to relax the equality constraint
because checking that $\rho_\infty(\mathcal W) = 1$ is
NP-hard~\cite{Jungers2009}. In contrast, checking the inequality
$\rho_\infty(\mathcal W) \leq 1$ can be efficiently done (using, for example,
the JSR software toolbox~\cite{Vankeerberghen2014}).
\end{remark}

In what follows, we compare the proposed lower bounds with the most relevant
existing bounds found in the literature. The first lower bound is the one
implicitly obtained by Zhou~\cite{Zhou1998a}; if we let
\begin{equation*}
\ellZ 
= 
\frac{\rho(\sum_{i=1}^N A_i\otimes A_i)}{{N}\rho_\infty}, 
\end{equation*}
then $\rho_1 \geq \ellZ$. This inequality can be derived from the following
bound on the joint spectral radius: $\rho_\infty \geq
({\rho_{p+q}}/{\rho_p})^{p/q} \rho_{p+q}$ whenever $p, q \geq 1$
\cite[p.~48]{Zhou1998a}. Letting $p=q=1$  gives $\rho_1 \geq
\rho_2^2/\rho_\infty$. Then, applying {assertion~2 of}
Proposition~\ref{prop:ease} to this inequality proves $\rho_1\geq \ellZ$. The
second lower bound was introduced in~\cite{Barthelemy2013a}: for $w_1, \dotsc,
w_N \in [-1,1]$ it holds that
\begin{equation*}
\rho_1 \geq \rho\left(\frac{1}{N} \sum_{i=1}^N w_i A_i\right).
\end{equation*}
In the following theorem, we show that the lower bound herein proposed extends the ones mentioned above:

\begin{theorem}\label{thm:5}
For each $m\geq 1$ define 
\begin{equation} \label{eq:ell_m}
\ell_m(\mathcal M)
= 
\sup \{
\lambda_{\mathcal W} 
: \mathcal W \in (\mathbb{R}^{m\times m})^N,\ \rho_\infty(\mathcal W) = 1
\}.
\end{equation}
The following statements are true:
\begin{enumerate}
\item If $m\geq n$, then $\ell_m \geq \ellZ$.

\item Define 
\begin{equation}\label{eq:def:ell_[-1,1]}
\ell_{[-1, 1]} 
= 
\max_{{w_1, \dotsc, w_N} \in [-1, 1]}
\rho\left(\frac{1}{N} \sum_{i=1}^N w_i A_i\right).
\end{equation}
Then $\ell_{[-1,1]} \leq \ell_m$ for every $m\geq 1$.
\end{enumerate}
\end{theorem}

\begin{proof}
We begin by {showing} that
\begin{equation}\label{eq:mm'}
[m <m']\ \Rightarrow\ [\ell_m \leq \ell_{m'}].
\end{equation} 
Let us take an arbitrary $\mathcal W \subset (\mathbb{R}^{m\times m})^N$ such
that $\rho_\infty (\mathcal W) = 1$. For each $i$ we define the block diagonal
matrix $V_i = \diag (W_i, O_{m'-m}) \in \mathbb{R}^{m'\times m'}$ and let
$\mathcal V = \{V_1, \dotsc, V_N\}$. Then, it holds that $\lambda_{\mathcal W} =
\lambda_{\mathcal V}$. Since $\rho_\infty(\mathcal V) = 1$, this implies that
$\lambda_{\mathcal W} \leq \ell_{m'}$. Finally, taking the {supremum} with
respect to $\mathcal W$ in the last inequality proves \eqref{eq:mm'}

Let us prove the theorem. It is clear that the normalized {family of matrices}
$\mathcal M_0 = \{A_1/\rho_\infty, \dotsc, A_N/\rho_\infty \} \subset
\mathbb{R}^{n\times n}$ has a joint spectral radius equal to one. Therefore,
from Theorem~\ref{thm:main} it follows that $\ell_n \geq \lambda_{\mathcal M_0}
= \rho( N^{-1}\sum_{i=1}^N A_i \otimes A_i)/{\rho_\infty} = \ellZ$. This
inequality and \eqref{eq:mm'} prove the first claim in the theorem. To prove the
second claim, we take an arbitrary {family} $\{w_1, \dotsc, w_N\} \subset
[-1,1]$. Since this {family} has a joint spectral radius less than or equal to
one, we obtain $\rho(N^{-1}\sum_{i=1}^N w_i A_i) = \lambda_{\{w_1, \dotsc,
w_N\}} \leq \ell_1$ by Theorem~\ref{thm:main} and Remark~\ref{rem:rho_infty<=1}.
Therefore, $\ell_{[-1,1]} \leq \ell_1$ and hence \eqref{eq:mm'} proves the
second claim. {We finally remark that the maximum in {\eqref{eq:def:ell_[-1,1]}}
exists because the set $[-1, 1]^N$ is compact and the function $\rho(\cdot)$ is
continuous.}
\end{proof}

In what follows, we illustrate our results with some examples.

\begin{example}
Let $N\geq 1$ be arbitrary. {Consider the matrix family $\mathcal M = \{ A_1,
\dotsc, A_{N+1}\}$, where} $A_1 = NI$ and
\begin{equation}\label{eq:IRRR}
A_2 = \cdots =
A_{N+1}  = R = \begin{bmatrix}
0 &-1 \\ 1 & 0
\end{bmatrix}.
\end{equation}
Since $\norm{R^k} = 1$ for every $k\geq 0$, we see that $\rho_1$ equals the
$1$-radius of the {family of scalars} $\{N, 1, \dotsc, 1\}${, where element~$1$
has multiplicity~$N$}. The $1$-radius of this {family} equals
\begin{equation*}
\rho\left(\frac{N + 1 + \cdots + 1}{N+1}\right) = \frac{2N}{N+1}
\end{equation*}
by {assertion~2 of} Proposition~\ref{prop:ease}. Therefore {$\rho_1(\mathcal M)
= (2N)/(N+1)$}. This value is attained by the proposed lower bound $\ell_2$
because
\begin{equation*}
\begin{aligned}
{\lambda_{\{I,\,R,\,\dotsc,\,R\}}}
&{= \rho\left(\frac{1}{N+1}\left(I\otimes (NI) + N R\otimes R \right)\right)}
\\
&
{=\frac{N}{N+1}\rho(I\otimes I + R\otimes R)}
\\
&
{=\frac{2N}{N+1}.}
\end{aligned}
\end{equation*}
On the other hand, a straightforward computation shows that $\ellZ = 1$.  
\end{example}

\begin{example}
Let $\mathcal M = \{I,\,R,\,R^2,\,R^3\}$, where $R$ is given in \eqref{eq:IRRR}.
Clearly, $\rho_1 = 1$ because $\norm{R^k} = 1$ for every $k\geq 0$. The proposed
lower bound attains this exact value of the $1$-radius as $\ell_2 = 1$ because
$\lambda_{I,\,R,\,R^2,\,R^3} = 1$. On the other hand, we can show that
$\ell_{[-1,1]} \leq \sqrt{2}/2$ as follows. Let $w_i\in [-1, 1]$ ($i=1,\dotsc,
4$) be arbitrary. Then we have
\begin{equation*}
S = \frac{w_1I + w_2R + w_3 R^2 + w_4 R^3}{4} =
\frac{1}{4} \begin{bmatrix}
u & -v\\v &u
\end{bmatrix}, 
\end{equation*}
where $u = w_1 - w_3$ and $v = w_2 - w_4$. Since $\abs u \leq 2$  and $\abs v
\leq 2$, we can see $\rho(S) \leq \sqrt 2 /2$ and therefore $\ell_{[-1,1]} \leq
\sqrt 2 /2$.
\end{example}

Although $\ell_m$ can largely improve other lower bounds in the literature, it
is not easy to compute. The first reason is the non-convexity of the
function~$\rho$ (see, e.g., \cite{Overton1988}). The other reason is that the
set $\{ \mathcal W \in (\mathbb{R}^{m\times m})^N :\rho_\infty(\mathcal W) = 1
\}$ does not admit an appropriate parametrization due to the
NP-hardness~\cite{Jungers2009} of computing $\rho_\infty$. For these reasons, we
here propose using the set of matrix weights from the following set:
\begin{equation*}
\mathfrak O_m 
= 
\left\{ 
\{D_i L_i\}_{i=1}^N
: 
\text{$D_i$ is diagonal, $\norm{D_i} \leq 1$, and $L_i$ is orthogonal}
\right\}.
\end{equation*}
Since the matrices in this set have a joint spectral radius less than or equal
to one, the best lower bound $\ell_{\mathfrak O_m} = \max\{\lambda_{\mathcal W}:
 \mathcal W \in \mathfrak O_m\}$\label{def:ell_Dm} achieved by using such matrix
weights provides a lower bound on the $1$-radius by Theorem~\ref{thm:main} and
Remark~\ref{rem:rho_infty<=1}. Moreover, since an orthogonal matrix~$L$ admits
\cite{Hurlimann2013} the parametrization $L = D(I-S)(I+S)^{-1}$ where $S$ is a
skew-symmetric matrix and $D$ is a diagonal matrix whose diagonals are either
$+1$ or $-1$, we can maximize $\ell_{\mathfrak O(m)}$ using, for example, a
stochastic gradient descent algorithm~\cite{Burke2005}. Notice that the result
of this algorithm will be a local maximum.

In the following example, we illustrate the effectiveness of the weights from
the set $\mathfrak O_m$ by studying the stability of a switched linear system.

\begin{example}\label{ex:}
Consider the switched linear system
\begin{equation}\label{eq:Xsystem:iid}
X(k+1) = A_{\sigma(k+1)}X(k),\ X(0) = I
\end{equation}
where $ \{\sigma(k)\}_{k=1}^\infty$ are random variables independently and
uniformly distributed on~$\{1, \dotsc, N\}$. The system is said to be {\it $p$th
mean stable} \cite{Kozin1969} if there exist $C>0$ and $\gamma \in [0, 1)$ such
that
\begin{equation}\label{eq:def:Sigma:pth-stbl}
E[\norm{X(k)}^p] \leq C\gamma^{pk}
\end{equation}
for every $k\geq 0$. It is well known \cite{Jungers2009} that 
\begin{equation}\label{eq:iidstblchar}
[\text{\eqref{eq:Xsystem:iid} is $p$th mean stable}] \ 
\Leftrightarrow\  [\text{$\rho_p < 1$}]. 
\end{equation}
In this example, we let $N=2$ and randomly choose two matrices 
\begin{equation*}
A_1  =
\begin{bmatrix}
-0.87&-0.77\\ 1.17&-1.09
\end{bmatrix}
,\ 
A_2 = 
\begin{bmatrix}
0.14&0.40\\ 0.89&-0.73
\end{bmatrix}.
\end{equation*}
Using a MATLAB implementation~\cite{Burke} of the stochastic gradient descent
algorithm, we perform the maximization over the set $\mathfrak O_m$ and find
\begin{equation*}
W_1  =
\begin{bmatrix}
-0.71&-0.70
\\ 
0.70&-0.71
\end{bmatrix}
,\ 
W_2 = 
\begin{bmatrix}
0.85&-0.53
\\ 
0.53&-0.85
\end{bmatrix}. 
\end{equation*}
Since $\lambda_{\{W_1,W_2\}} = 1.07$ we obtain $\ell_2 \geq 1.07$. Thus,
Theorem~\ref{thm:main} implies that $\rho_1 > 1$ and hence the the system
in~\eqref{eq:Xsystem:iid} is unstable, according to \eqref{eq:iidstblchar}. We
cannot prove that the system in~\eqref{eq:Xsystem:iid} is unstable using the
other lower bounds in the literature, since $\ellZ = 0.93$ and $\ell_{[-1,1]} =
0.73$. We remark that the joint spectral radius appearing in $\ellZ$ is
evaluated with the JSR Toolbox \cite{Vankeerberghen2014}. Also, the maximum in
$\ell_{[-1,1]}$ has been evaluated with extensive simulations on the weights.
\end{example}

We propose a further extension of our bounds based on product families of a set
of matrices. Let us define ${\mathcal M}^q$ to be the {family of matrices
consisting} of all the {$N^q$} products of the matrices from $\mathcal M$ having
length $q$. Then, our extension can be stated as follows:

\begin{theorem}
Let $m$ and $q$ be positive integers. Define $\ell^{(q)}_m(\mathcal M)
= \ell_m(\mathcal M^q)^{1/q}$. Then, $\ell^{(q)}_m \leq \rho_1$.
Furthermore, if $q$ is a divisor of another positive integer $q'$,
then $\ell^{(q)}_m \leq \ell^{(q')}_m$.
\end{theorem}

\begin{proof}
Let us first recall the following identities~\cite{Jungers2009}:
\begin{align}\label{eq:rho_p(M^k)}
\rho_p(\mathcal M^q) &= \rho_p(\mathcal M)^q, 
\\
\label{eq:rho_infty(M^k)}
\rho_\infty(\mathcal M^q) &= \rho_\infty(\mathcal M)^q.
\end{align}
Then, using Theorem~\ref{thm:main} and Equation~\eqref{eq:rho_p(M^k)}, we can
prove the first claim in the theorem as $\rho_1(\mathcal M) = \rho_1(\mathcal
M^q)^{1/q} \geq \ell_m(\mathcal M^q)^{1/q} = \ell_m^{(q)}(\mathcal M)$. To prove
the second claim in the theorem, we let $q=1$ and $q'=2$ for simplicity. The
proof for general $q$ and $q'$ is similar to this particular case and hence is
omitted. Let $\mathcal W = \{W_1, \dotsc, W_N\} \subset \mathbb{R}^{m\times m}$
be arbitrary and assume $\rho_\infty(\mathcal W) = 1$. Since $\rho(M)^2 =
\rho(M^2)$ for a square matrix $M$, we can show that
\begin{equation}\label{eq:pf:(q)}
\begin{aligned}
\lambda_{\mathcal W}(\mathcal M)^2
&=
\rho\left(
\left(\frac{1}{N}\sum_{i=1}^N W_i \otimes A_i
\right)^{2}\right)
\\
&=
\rho \left(
\frac{1}{N^2}{\sum_{i, j = 1}^N (W_i W_{\!j}) \otimes (A_i A_j)}
\right)
\\&
=
\lambda_{\mathcal W^2}(\mathcal M^2).
\end{aligned}
\end{equation}
Also, since $\rho_\infty(\mathcal W^2) = \rho_\infty(\mathcal W)^2 = 1$ by
\eqref{eq:rho_infty(M^k)}, {it follows that} $\lambda_{\mathcal W^2}(\mathcal
M^2) \leq \ell(\mathcal M^2) = \ell^{(2)}(\mathcal M)^2$. This inequality and
\eqref{eq:pf:(q)} yield $\lambda_{\mathcal W} \leq \ell^{(2)}$. Taking the
{supremum} with respect to $\mathcal W$ in the left-hand side of this inequality
proves $\ell_m^{(1)} \leq \ell^{(2)}_m$, as desired.
\end{proof}

We close this section by giving a remark on complex weights. 

\begin{remark}\label{rem:}
In principle, one could obtain better lower bounds using complex weights in
$\ell_m$ instead of the real weights $W_1, \dotsc, W_N$. However, we can show
that it does not lead to an essential improvement. Precisely speaking, we here
prove the following claim: if $\mathcal W = \{W_1, \dotsc, W_N\} \subset
\mathbb{C}^{m\times m}$ satisfies $\rho_\infty(\mathcal W) = 1$, then
\begin{equation}\label{eq:complexweights}
\lambda_{\mathcal W} \leq \ell_{2m}.
\end{equation}
To prove this claim, for $W \in \mathbb{C}^{m\times m}$ we let
\begin{equation*}
T_W = \begin{bmatrix}
\real W&-\imag W\\\imag W& \real W
\end{bmatrix} \in \mathbb{R}^{(2m) \times (2m)}, 
\end{equation*}
where $\real W$ and $\imag W$ denotes the real and imaginary parts of $W$,
respectively. The multiplicative property $T_{WW'} = T_W T_{W'}$ and the
identity $\norm{T_W} = \norm{W}$ yield
\begin{equation}\label{eq:JSR:TW}
\rho_\infty(T_{W_1}, \dotsc, T_{W_N}) = 1. 
\end{equation}
Also, {since} $\rho(W) = \rho(T_W)$, we can show that 
\begin{equation*}
\begin{aligned}
\lambda_{\mathcal W}(\mathcal M)
&=
\rho\left(
T_{N^{-1}\sum_{i=1}^N W_i \otimes A_i } 
\right)
\\
&=
\rho\left(\frac{1}{N}
\begin{bmatrix}
\sum_{i=1}^N (\real W_i) \otimes A_i & \sum_{i=1}^N -(\imag W_i) \otimes A_i
\\
\sum_{i=1}^N (\imag W_i) \otimes A_i & \sum_{i=1}^N (\real W_i) \otimes A_i
\end{bmatrix}
\right)
\\
&=
\rho\left(\frac{1}{N}\sum_{i=1}^N
T_{W_i} \otimes A_i\right)
\\
&=
\lambda_{\{T_{W_1}, \dotsc, T_{W_N}\}}(\mathcal M).
\end{aligned}
\end{equation*}
This equation and \eqref{eq:JSR:TW} prove Inequality~\eqref{eq:complexweights}.
\end{remark}

\section{The Markovian case}\label{sec:markov}

In this section, we extend the results presented in the last section to the
Markovian case. Let $\sigma = \{\sigma(k)\}_{k=1}^\infty$ be a time-homogeneous
Markov chain with a state space~$\{1, \dotsc, N\}$ and a transition probability
matrix~$Q\in \mathbb{R}^{N\times N}$. We define the Markovian version of the
$p$-radius as follows.

\begin{definition}\label{defn:}
Let $X(\cdot;\mu)$ denote the trajectory of the Markov jump linear system
\begin{equation*}
\Sigma:  X(k+1) = A_{\sigma(k+1)}X(k),\ X(0) = I,\,\sigma(1) \sim \mu, 
\end{equation*}
where $\mu$ is an arbitrary probability distribution on $\{1, \dotsc, N\}$. The
{\it $L^p$-norm Markovian joint spectral radius} ({\it Markovian $p$-radius} for
short) of the pair $(\mathcal M, Q)$ is defined by
\begin{equation*}
\rho_p(\mathcal M, Q)
=
\sup_{\mu} \limsup_{k\to\infty}\left(E[\norm{X(k;\mu)}^p]^{1/(pk)}\right). 
\end{equation*}
\end{definition}

\begin{remark}
The logarithm of the Markovian $p$-radius corresponds to a quantity called
Lyapunov exponent of the $p$th mean \cite[p.~307]{Fang1995}.
\end{remark}

We will later see that this definition coincides with the originally defined
$p$-radius in the case where the matrices~$A_{\sigma(k)}$ are independent and
uniformly distributed at every time step. Moreover, as can be naturally
expected, the Markovian $p$-radius has a close connection with the mean
stability of $\Sigma$, which is defined as follows. We say that $\Sigma$ is {\it
$p$th mean stable} if there exist $C>0$ and $\gamma \in [0, 1)$ such that
\eqref{eq:def:Sigma:pth-stbl} holds for all $k$ and $\mu$. The next proposition
is an immediate consequence from the definition of the Markovian $p$-radius.

\begin{proposition}\label{prop:markovstblchar}
$\Sigma$ is $p$th mean stable if and only if $\rho_p(\mathcal M, Q) < 1$. 
\end{proposition}

\begin{proof}
The necessity is obvious. Let us prove the sufficiency. Assume that
$\rho_p(\mathcal M, Q) $ $< 1$. Then there exists a $\gamma \in [0, 1)$ such
that, for every \mbox{$i\in \{1, \dotsc, N\}$}, we have $\limsup_{k\to\infty}
(E[\norm{X(k;\delta_i)}^p]^{1/(pk)}) < \gamma$, where $\delta_i$ is the
probability distribution on $\{1, \dotsc, N\}$ such that $\delta_i(\{i\}) = 1$.
Then, for each $i$ there exists a positive integer $K_i$ such that, if $k>K_i$,
then $E[\norm{X(k;\delta_i)}^p] < \gamma^{pk}$. Therefore, for an arbitrary
$\mu$, if $k>\max(K_1, \dotsc, K_N)$ then $E[\norm{X(k;\mu)}^p] = \sum_{i=1}^N
\mu_i E[\norm{X(k;\delta_i)}^p] \leq \sum_{i=1}^N \mu_i \gamma^{pk} =
\gamma^{pk}$. This implies that $\Sigma$ is $p$th mean stable.
\end{proof}

The main result in this section is stated in the following theorem, which can be
used to compute upper and lower bounds on the Markovian $p$-radius.

\begin{theorem}\label{thm:main:Markov}
\begin{enumerate}
\item The sequence $\{h_k(\mathcal M, Q)\}_{k=1}^\infty$ defined by 
\begin{equation*}
h_k(\mathcal M, Q) 
= 
\left(
\sum_{i_1,\dotsc,i_k =1}^N 
q_{i_1,i_2} \dotsm q_{i_{k-1},i_k}
\norm{A_{i_k} \dotsm A_{i_1}}^p\right)^{1/(pk)}
\end{equation*}
is decreasing. Moreover $\rho_p(\mathcal M, Q) = \lim_{k\to\infty} h_k(\mathcal M, Q)$. 

\item For $\bar{\mathcal W} = \{W_{ij}\}_{1\leq i, j\leq N} \subset
\mathbb{R}^{m\times m}$ define
\begin{equation*}
{\mathcal A}_{\bar{\mathcal W}} = 
\begin{bmatrix}
q_{11} W_{11}\otimes A_1 & \cdots & q_{N1}W_{N1}\otimes A_N
\\
\vdots &\ddots&\vdots
\\
q_{1N} W_{1N}\otimes A_1 & \cdots & q_{NN}W_{NN}\otimes A_N
\end{bmatrix}. 
\end{equation*}
If $\rho_\infty(\bar{\mathcal W}) = 1$, then $\rho_{1}(\mathcal M, Q) \geq
\rho({\mathcal A}_{\bar{\mathcal W}})$.
\end{enumerate}
\end{theorem}

Let us observe some consequences of Theorem~\ref{thm:main:Markov} before proving
it. First, consider the special case when $\sigma$ is a sequence of independent
and uniformly distributed random variables on $\{1, \dotsc, N\}$. The
corresponding transition probability matrix $Q$ is the $N\times N$ matrix whose
entries are all $1/N$. In this case, we have $h_k(\mathcal M, Q)
=${$(\sum_{i_1,\dotsc,i_k =1}^N ({1/N})^{k-1} \norm{A_{i_k} \dotsm
A_{i_1}}^p)^{1/(pk)}$} $=N^{1/kp}h_k(\mathcal M)$. {This equation implies that
$h_k(\mathcal M, Q)$, the average of the norm of $k$-product $A_{i_k}\dotsm
A_{i_1}$, coincides with the other average $h_k(\mathcal M)$ under the
identification of $\Sigma$ as a switched linear system having independent and
identically distributed system parameters, except the factor $N^{1/kp}$. This
factor, roughly speaking, arises because $h_{k}(\mathcal M, Q)$ does not take
the initial probability distribution of the switching signal $\sigma$ into
account.} {Then, by this equation and} the first claim of
Theorem~\ref{thm:main:Markov}, taking the limit {as} $k\to\infty$ shows that
$\rho_p(\mathcal M, Q) = \rho_p(\mathcal M)$. Hence Definition~\ref{defn:}
indeed recovers the original $p$-radius.

Next, as a corollary of Theorem~\ref{thm:main:Markov}, we can recover a lower
bound of the Markovian $p$-radius implicitly presented in~\cite{Barthelemy2013a}.

\begin{corollary}[\cite{Barthelemy2013a}]
Define
\begin{equation*} \label{eq:def:ellm(M,Q)}
\ell_m(\mathcal M, Q)
= 
{\sup} \left\{
\rho({\mathcal A}_{\bar{\mathcal W}}) : \bar{\mathcal W} \in (\mathbb{R}^{m\times m})^{N^2},\ \rho_\infty(\bar {\mathcal W}) = 1
\right\}.
\end{equation*}
Then, $\ell_1(\mathcal M, Q) \leq \rho_{1}(\mathcal M, Q)$.
\end{corollary}

Before we prove Theorem~\ref{thm:main:Markov}, we illustrate its use with an example.

\begin{example}
Consider the Markov jump linear system~$\Sigma$ with
\begin{equation*}
\begin{gathered}
A_1  =
\begin{bmatrix}
0.77&0.80\\ -0.60&0.87
\end{bmatrix}
,\ 
A_2 = 
\begin{bmatrix}
-0.77&0.83\\ -0.70&-0.70
\end{bmatrix}
,
\\
Q = 
\begin{bmatrix}
0.70&0.30\\ 0.43&0.57
\end{bmatrix}. 
\end{gathered}
\end{equation*}
A brute force search indicates that $\ell_{1}(\mathcal M, Q) = 0.844$ with the
weights $W_{11} = 1$, $W_{12} = 1$, $W_{21} = -1$, and $W_{22} = 0.932$. On the
other hand, using a MATLAB implementation~\cite{Burke} of the stochastic
gradient descent algorithm used in Example~\ref{ex:}, we locally maximize
$\ell_{2}(\mathcal M, Q)$ over the set $\mathfrak O_2$ of matrix weights to find
the set $\bar{\mathcal W}$ consisting of the following matrices:
\begin{equation*}
\begin{gathered}
W_{11}  =
\begin{bmatrix}
-0.412&-0.911
\\ 
0.911&-0.412
\end{bmatrix}
,\ 
W_{12} = 
\begin{bmatrix}
0.839&-0.544
\\ 
0.544&-0.839
\end{bmatrix}
, \\
W_{21}  =
\begin{bmatrix}
-0.204&-0.979
\\ 
0.979&-0.204
\end{bmatrix}
,\ 
W_{22} = 
\begin{bmatrix}
0.937&-0.349
\\ 
0.349&-0.937
\end{bmatrix}.
\end{gathered} 
\end{equation*}
Since $\rho(\mathcal A_{\bar{\mathcal W}}) = 1.067$, we conclude
$\rho_1(\mathcal M, Q) \geq 1.067$. This proves that the corresponding Markov
jump linear system is unstable.
\end{example}

The rest of this section is devoted to the proof of
Theorem~\ref{thm:main:Markov}. In our proof, we employ a reduction of the
Markovian $p$-radius to the original $p$-radius, for which we can apply the
results obtained in Section~\ref{sec:NovelLBB}.

\begin{proposition}\label{prop:markov:char}
Let $p$ be a positive integer. For each $1\leq i\leq N$, let $e_i$ denote the
$i$th vector in the canonical basis of $\mathbb{R}^N$. Define the set of
matrices $\bar{\mathcal M}_p = \{B_{ij}^{(p)}\}_{i, j\in \{1, \dotsc, N\}}$ as
\begin{equation*}
B_{ij}^{(p)} = N^{2/p} q_{ij}^{1/p} A_i\otimes (e_j e_i^\top). 
\end{equation*}
Then,
\begin{equation}\label{eq:prop:markov:char}
\rho_p(\mathcal M, Q) = \rho_p(\bar {\mathcal M}_p). 
\end{equation}
\end{proposition}

\begin{proof}
{We first claim that, for the proof of \eqref{eq:prop:markov:char}, it is
sufficient to show that}
\begin{equation}\label{eq:equiv}
[\text{$\rho_p(\mathcal M, Q) < 1$}]
\ \Leftrightarrow \ 
\text{[$\rho_p(\bar{\mathcal M}_p) < 1$],}
\end{equation}
due to the following reason. Suppose that \eqref{eq:prop:markov:char} does not
hold while \eqref{eq:equiv} is true. Then, we have either $\rho_p(\mathcal M, Q)
> \rho_p(\bar {\mathcal M}_p)$ or $\rho_p(\mathcal M, Q) < \rho_p(\bar {\mathcal
M}_p)$. If the former inequality holds, then one can find a $c>0$ such that the
matrix family $c\mathcal M = \{cA_i\}_{i=1}^N$ satisfies $\rho_p(c\mathcal M, Q)
> 1 > \rho_p(c \bar {\mathcal M}_p)= \rho_p( \overline{c \mathcal M}_p)$.
However, this cannot be true by \eqref{eq:equiv}. In a similar way, we can also
show that the latter inequality, $\rho_p(\mathcal M, Q) < \rho_p(\bar {\mathcal
M}_p)$, cannot hold. Therefore, \eqref{eq:prop:markov:char} must be true.

{Consequently, it suffices to prove \eqref{eq:equiv}.} In order to prove the
claim in~\eqref{eq:equiv}, we introduce an alternative switched linear system
with independent and identically distributed jumping parameters. Let
$\{\theta(k)\}_{k=1}^\infty$ and $\{\phi(k)\}_{k=1}^\infty$ be independent
random variables uniformly distributed on $\{1, \dotsc, N\}$. Define the
switched linear system $\bar \Sigma_p$ by
\begin{equation*}
\bar \Sigma_p
: 
\bar X(k+1) = B_{\theta(k+1), \phi(k+1)}^{(p)} \bar X(k),\ \bar X(0) = I.
\end{equation*}
From \eqref{eq:iidstblchar}, we see that $\rho_p(\bar{\mathcal M}_p) < 1$ if and
only if $\bar{\Sigma}_p$ is $p$th mean stable. Moreover, from
Proposition~\ref{prop:markovstblchar}, we know that $\rho_p(\mathcal M, Q) < 1$
if and only if $\Sigma$ is $p$th mean stable. Therefore, to prove
\eqref{eq:equiv}, we need to show that $\Sigma$ is $p$th mean stable if and only
if $\bar \Sigma_p$ is $p$th mean stable.

To prove the equivalence of stability, let us first compute $E[\norm{\bar
X(k)}^p]$. By the definition of $\bar{\Sigma}_p$ and Equation \eqref{eq:ABCD},
we can compute $\bar X(k)$ as
\begin{equation*}
\begin{aligned}
\bar X(k)
&=
\prod_{i=1}^k \left(
N^{2/p} q_{\theta(i), \phi(i)}^{1/p} A_{\theta(i)} \otimes  (e_{\phi(i)} e_{\theta(i)}^\top)
\right)
\\
&=
N^{2k/p} 
\left(q_{\theta(1),\phi(1)} \dotsm q_{\theta(k),\phi(k)}\right)^{1/p} 
(A_{\theta(k)} \dotsm A_{\theta(1)})\otimes J_k, 
\end{aligned}
\end{equation*}
where the symbol $\prod_{i=1}^k$ denotes the left product of matrices and also
$J_k = e_{\phi(k)}e_{\theta(k)}^\top \dotsm e_{\phi(1)} e_{\theta(1)}^\top$.
Since the vectors $e_1, \dotsc, e_N$ are orthonormal, $\norm{J_k} = 1$ if
\begin{equation}\label{eq:event}
\phi(i) = \theta(i+1),\ i=1, \dotsc, k-1,
\end{equation}
and otherwise $\norm{J_k} = 0$. Therefore, if we denote by $\chi$ the
characteristic function of the event~\eqref{eq:event}, using \eqref{eq:|AoxB|}
we obtain
\begin{equation*}
\norm{\bar X(k)}^p 
= 
\chi N^{2k} q_{\theta(1),\theta(2)} \dotsm q_{\theta(k-1),\theta(k)} q_{\theta(k),\phi(k)} 
\norm{A_{\theta(k)} \dotsm A_{\theta(1)}}^p.
\end{equation*}
Since the event in \eqref{eq:event} occurs with probability $1/N^{k-1}$, we obtain
\begin{equation}\label{eq:E[normY(k)^p]}
\begin{aligned}
&E[\norm{\bar X(k)}^p] 
\\
=&\,
\frac{N^{2k}}{N^{k-1}}
E\left[q_{\theta(1),\theta(2)} \dotsm q_{\theta(k-1),\theta(k)} q_{\theta(k),\phi(k)} 
\norm{A_{\theta(k)} \dotsm A_{\theta(1)}}^p
\right]
\\
=&\,
N^{k+1}\sum_{i_1, \dotsc, i_k, j = 1}^N
\frac{1}{N^{k+1}} 
q_{i_1,i_2} \dotsm q_{i_{k-1},i_k} q_{i_k,j} 
\norm{A_{i_k} \dotsm A_{i_1}}^p
\\
=&\,
\sum_{i_1,\dotsc,i_k =1}^N q_{i_1,i_2} \dotsm q_{i_{k-1},i_k}
\norm{A_{i_k} \dotsm A_{i_1}}^p
\\
=&\,
\sum_{i_1=1}^N E[\norm{X(k; \delta_{i_1})}^p], 
\end{aligned}
\end{equation}
where we used $\sum_{j=1}^N q_{i_k, j} = 1$ to show the third equality.

Now, assume that $\Sigma$ is $p$th mean stable. Then, by
\eqref{eq:E[normY(k)^p]} we have $E[\norm{\bar X(k)}^p]$ $\leq \sum_{i_1=1}^N
C\gamma^{pk} = CN\gamma^{pk}$ for some $C>0$ and $\gamma\in [0, 1)$, and hence
$\bar{\Sigma}_p$ is $p$th mean stable. On the other hand, assume that $\bar
\Sigma_p$ is $p$th mean stable. Then, there exist $C>0$ and $\gamma\in [0, 1)$
such that $E[\norm{\bar X(k)}^p] \leq C\gamma^{pk}$. Therefore,
\eqref{eq:E[normY(k)^p]} shows that $E[\norm{X(k; \delta_{i_1})}^p] \leq
C\gamma^{pk}$ for every $i_1 \in \{1, \dotsc, N\}$. Then, in the same way as the
proof of Proposition~\ref{prop:markovstblchar}, we can conclude the $p$th mean
stability of $\Sigma$. This completes the proof of the theorem.
\end{proof}

\begin{remark}
From Proposition~\ref{prop:markov:char}, we can regard the matrix $B_{ij}^{(p)}$
as an $L^p$-averaged version of the $\Omega$-lift introduced in
\cite{Kozyakin2014}, which is used to generalize the so-called Berger-Wang
formula to a Markovian version of the joint spectral radius. Also, we remark
that considering the auxiliary switched linear system~$\bar \Sigma_p$ with an
extended state space is similar to considering an extended state-variable
consisting of the original state variable and the underlying Markov chain, which
is frequently employed for studying (semi-)Markov jump linear
systems~\cite{Costa2005} (\cite{Ogura2013f}).
\end{remark}

Finally, let us prove Theorem~\ref{thm:main:Markov}.

\begin{proof}[Proof of Theorem~\ref{thm:main:Markov}]
From Equation \eqref{eq:E[normY(k)^p]}, we see that 
\begin{equation*}
h_k(\bar{\mathcal M}_{p})
= E[\norm{\bar X(k)}^p]^{1/(pk)} = h_k(\mathcal M, Q).
\end{equation*}
Therefore, by Proposition~\ref{prop:ease}, the sequence $\{h_k(\mathcal M,
Q)\}_{k=1}^\infty$ is decreasing. Furthermore, it converges to
$\lim_{k\to\infty} h_k(\bar{\mathcal M}_{p}) = \rho_p(\bar {\mathcal M}_p) =
\rho_p(\mathcal M, Q)$ by \eqref{eq:prop:markov:char}. Thus, the first claim is
proved. Let us then prove the second claim. Assume that $\bar{\mathcal W} =
\{W_{ij}\}_{1\leq i, j\leq N} \subset \mathbb{R}^{m\times m}$ satisfies
$\rho_\infty(\bar{\mathcal W}) = 1$. By \eqref{eq:prop:markov:char}, it is
enough to show that $\rho_1(\bar{\mathcal M}_1) \geq \rho({\mathcal
A}_{\bar{\mathcal W}})$. From Theorem~\ref{thm:main} we obtain
\begin{equation}\label{eq:lobbpre}
\begin{aligned}
\rho_1(\bar{\mathcal M}_1)
&\geq
\rho\left(
\frac{1}{N^2}\sum_{i, j =1}^N W_{ij} \otimes B_{ij}^{(1)}\right)
\\
&=
\rho\left(
\sum_{i, j = 1}^N q_{ij} W_{ij} \otimes A_i \otimes (e_j e_i^\top)\right). 
\end{aligned}
\end{equation}
Here we recall that there exists~\cite{Brewer1978} an invertible matrix~$T$
satisfying $C\otimes D = T^{-1}(D\otimes C)T$ for all $C\in \mathbb{R}^{N\times
N}$ and $D\in \mathbb{R}^{(nm)\times (nm)}$. {Consequently, there exists an
invertible matrix $T$ such that} $W_{ij} \otimes A_i \otimes (e_j e_i^\top) =
T^{-1} ((e_j e_i^\top)\otimes W_{ij} \otimes A_i)T$. Therefore, the matrix
appearing in the last term of \eqref{eq:lobbpre} is similar to $\sum_{i, j=1}^N
q_{ij}(e_j e_i^\top) \otimes W_{ij} \otimes A_i$, which in fact equals
${\mathcal A}_{\bar{\mathcal W}}$. Therefore $\rho_1(\bar{\mathcal M}_1) \geq
\rho({\mathcal A}_{\bar{\mathcal W}})$, as desired.
\end{proof}

\section{Conclusion}

This paper proposed novel lower bounds on the $p$-radius of a finite set of
matrices. The obtained lower bound is given by the spectral radius of an average
of the given matrices weighted via Kronecker products. We showed that the
proposed lower bounds theoretically extend and also practically improve the
existing lower bounds. We have also shown the extension of the $p$-radius and
its lower bounds to the Markovian case.

\section*{Acknowledgment}

This work was supported in part by the NSF under grants CNS-1302222 and
IIS-1447470. R.J. is an FNRS Research Associate. His work is supported by the
Communaut\'e fran\c caise de Belgique (ARC), and by the Belgian state (PAI).
Parts of this work were carried out while the first author was visiting the
Institute of ICTEAM (Information and Communication Technologies, Electronics and
Applied Mathematics) at Universit\'e catholique de Louvain. He would like to
acknowledge the generous hospitality of the institute.

\end{document}